\documentclass[11pt]{article}
\usepackage{amsmath,amssymb,amsbsy,amsfonts,amsthm,latexsym,amsopn,amstext,amsxtra,euscript,amscd,color}
\usepackage{mathtools}\usepackage{epsfig}
\usepackage{pdfpages}
 \usepackage[draft]{todonotes}
\usepackage{tikz}
\newcommand*\circled[1]{\tikz[baseline=(char.base)]{\node[shape=circle,draw,inner sep=2pt] (char) {#1};}}
\def\n{\noindent}
\newtheorem{theorem}{Theorem}
\newtheorem{thm}[theorem]{Theorem}

\newtheorem{proposition}[theorem]{Proposition}

\def\eproof{\hfill$\Box$}
\begin{document}

\title{\huge\bf \textrm{A variation on bisecting the binomial coefficients} }

\author{\Large Eugen J. Iona\c scu$^1$\\
\vspace{0.1cm} \\
\small $^1$Department of Mathematics, \\
\small Columbus State University\\
\small Columbus, GA 31907\\
}

\date{\today}
\maketitle
\thispagestyle{empty}

\begin{abstract}
In this paper, we  present an algorithm
which allows us to search for all the  bisections for the binomial coefficients $\{\binom{n}{k} \}_{k=0,...,n}$ and include a table with the results for all $n\le 154$.
Connections with previous work on this topic is included. We conjecture that the probability of having only trivial solutions is $5/6$.
\end{abstract}
{\bf Keywords:} Binomial coefficients, algorithm, congruence.

\section{Introduction}

This investigation is about the ubiquitous entries of the Pascal Triangle:

\begin{center}
\begin{table} [ht]
\caption{ Binomial Coefficients}{}
\vskip.2cm
\centering
\tiny $$\begin{tabular}{rccccccccccccccccccc}
$n=0$:&&&&&  &    &    &    &    &  1\\\noalign{\smallskip\smallskip}
$n=1$:&&&&&  &    &    &    &  1 &    &  1\\\noalign{\smallskip\smallskip}
$n=2$:&&&&&  &    &    &  1 &    &  2 &    &  1\\\noalign{\smallskip\smallskip}
$n=3$:&&&&&  &    &  1 &    &  3 &    &  3 &    &  1\\\noalign{\smallskip\smallskip}
$n=4$:&&&&&  &  1 &    &  4 &    &  6 &    &  4 &    &  1\\\noalign{\smallskip\smallskip}
$n=5$:&&&&&1 &   &  5  &    & 10 &    & 10 &    &  5 &  & 1\\\noalign{\smallskip\smallskip}
$n=6$:&&&&1  &   &  6  &    & 15 &    & 20 &    &  15 &  & 6 &  & 1\\\noalign{\smallskip\smallskip}
$n=7$:&&&1&  &  7 &    &  21  &  & 35   &  & 35   &   &21  &  &  7& & 1\\\noalign{\smallskip\smallskip}
$n=8$:&&1&& \boxed{-8}  &   & \boxed{-28}   &    & \circled{$\pm 56$} &   & 70 &   &  \circled{$\mp 56$} &   & \boxed{-28} &  & \boxed{-8} & & 1\\\noalign{\smallskip\smallskip}
$n=9$:&1&&9&  &  36 &   & 84   &  & 126  &  &  126 &   & 84  &  &  36&  & 9 &  & 1 \\\noalign{\smallskip\smallskip}
\end{tabular}$$
\label{table:binc}
\end{table}
 \end{center}

Our problem originates  in the early 1990's in a series of papers (\cite{GR97}, \cite{Jef91}, \cite{Mi90}, and \cite{NS}).  First,  Nisan and Szegedy (in \cite{NS}) began looking into polynomial functions (of one or more variables) which represent Boolean functions and their interest was to characterize the degree of such a polynomial (uniquely determined under certain conditions), and describe when is the degree the smallest possible (over the class of Boolean functions).  Around the same time, J. von zur Gathen and J. Rouche  (see \cite{GR97}), learned from Professor Mario Szegedy about the problem and concentrated on symmetric Boolean functions. Although the papers appeared at some distance in time, the two pairs of authors were aware of each other results years in advance. Our work is closely related with this last paper and we will point out the overlap and the new information.

The type of non-constant symmetric Boolean functions, which can be represented by a polynomial of only one variable, defined on $\{0,1,...,n\}$,  with the degree less than the expected one, namely $n$ (using Lagrange interpolation),  became of special interest for obtaining various cryptographic properties (see  Gopalakrishnan et al.~\cite{GHS93},  Cusick and Li~\cite{CL05}, Mitchell ~\cite{Mi90}, Sarkar and Maitra~\cite{Sarkar2007} and more recent works such as Castro, Gonzalez and Medina~\cite{CastroGonzalezMedina}).  A symmetric Boolean function with this special property is now referred to as  {\it balanced}. As is turns out, the existence of these functions is equivalent to what we refer here by the {\it binomial coefficients bisection problem} described below. The values of $n$, and the number of non-constant symmetric Boolean functions in $n$ variables,  with these special properties became of interest. In \cite{IonandCo}, the authors prove various bounds for the number of functions and show the connection of the problem with  the sequence A200147 which is basically the starting point of Nisan and Szegedy (in \cite{NS}) without the language of polynomials.

This problem, that we already alluded to, or simply (BCBP), is finding solutions \break $[\delta_0,\delta_1,\ldots,\delta_{n-1} ,\delta_n]\in\{-1,1\}^{n+1}$ of the equation

\begin{equation}\label{BCBP} \displaystyle \sum_{k=0}^n \delta_k\binom{n}{k}=0.\end{equation}

  The number of all  solutions of the (BCBP) is denoted by $J_n$. The sequence $\{J_n\}$ was introduced and studied in \cite{IonandCo}. It was shown that  $\{J_n\}$ is the same as the number of $0$'s or $1$'s arrays, of $n+1$ elements, with zero $n$-difference, which is recorded as the sequence A200147 in the The On-Line Encyclopedia of Integer Sequences. This identification is clearly the approach from \cite{GR97}, and it makes the objects studied here interesting from an analysis point of view.

The binomial theorem  gives $\sum_{k=0}^n (-1)^n\binom{n}{k}=(1-1)^n=0$, which shows  that \break $\pm [1,-1,1,-1,\ldots ]$ is always a solution of (BCBP), i.e., we have at least two solutions for every~$n$~($J_n\ge 2$).
We also observe  that if  $n$ is odd then $$[\delta_0,\ldots,\delta_{(n-1)/2},-\delta_{(n-1)/2},\ldots,-\delta_0]$$ with $\delta_i\in \{-1,1\}$ arbitrarily chosen, give $2^{(n+1)/2}$ solutions. All these are considered {\em trivial} solutions of (BCBP) (see \cite{CL05}). Our concept of a {\it non-trivial solution} is going to be a little different, in the way we will represent them,
and we will connect the old ideas with this one in the next section. Cusick and Li (\cite{CL05}) raised some questions about the set of values $n$ so that only the trivial solutions 
 of (BCBP) exist. Theorem 2.6 from \cite{GR97},  was rediscovered in  \cite{IonandCo}, and this provided a positive answer to questions  Q2 in Q4 (\cite{CL05}, page 86). Question Q1 is still open, and we want to include it explicitly:
 
 {\bf Q1:} \texttt{Are there infinitely many odd values of $n$ for which only the trivial \break solutions
 of (BCBP) exist?}
 
We believe the answer to this question is definitely Yes, but this situation is way harder than the even case. There are nontrivial solutions for the  (BCBP)  but they do not seem to be that many.
In fact, we conjecture that, except for the case $n\equiv 2$ (mod 6), the probability for the existence of non-trivial solutions is zero.
For $n=6k+2$, $k\in \mathbb N$, we always have non-trivial solutions and the first one  (for $n=8$) is $\Delta:=[1, -1, -1, 1, 1, -1, -1, -1, 1]$. In \cite{GR97}, the authors are interested the Lagrange polynomial representing the data above but just slightly altered, by changing the signs of every other value in the list $\Delta$, i.e., $\widetilde{\Delta}:=[1, 1, -1, -1, 1, 1, -1, 1, 1]$ on the domain $\{0,1,2,\cdots,8\}$ (see Figure~1):

$$\underset{Figure\ 1,\  \text{Lagrange polynomial $P$ interpolating}  \widetilde{\Delta} } {\epsfig{file=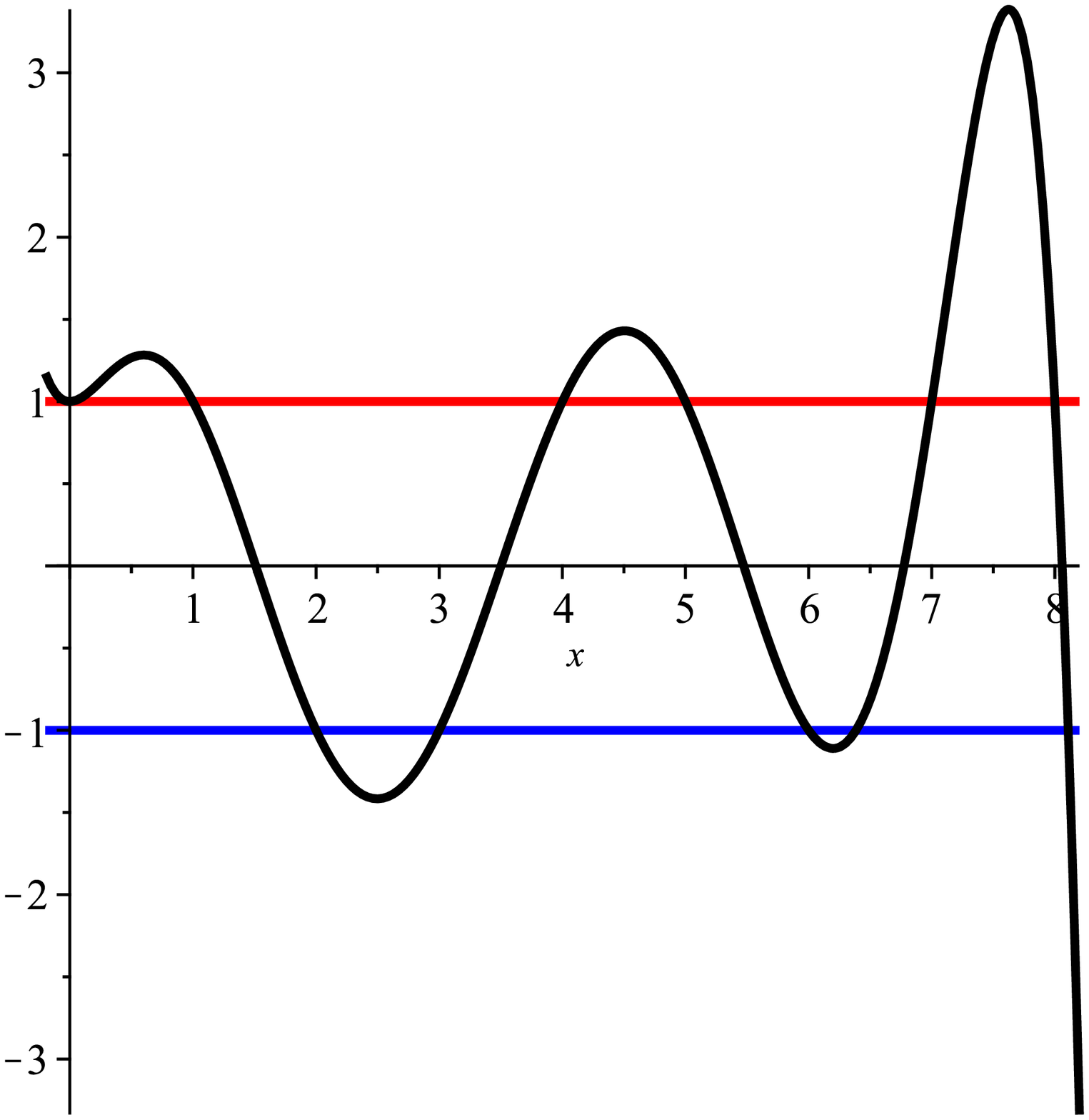,height=2in,width=3in} }$$
It turns out that this polynomial has degree $7$ (not $8$ as expected):

$$P(x)=1+\frac{28}{9}x^2-\frac{481}{90}x^3+\frac{203}{72}x^4-\frac{47}{72}x^5+\frac{5}{72}x^6
-\frac{1}{360}x^7.$$
In fact, the following result is discussed in \cite{GR97} (Theorem 2.2).

\vspace{0.1in}
\n {\bf Theorem:}  \texttt{The Lagrange polynomial interpolating the data $[\delta_0,-\delta_1, \ldots, (-1)^n \delta_n]$ on $\{0,1,...,n\}$ has degree less than $n$, if and only if
$[\delta_0,\delta_1, \ldots,  \delta_n]$ is a solution of (BCBP).}
\vspace{0.1in}

Moreover, they show that the degree of such a polynomial is less or equal than $n-r$, $r\ge 1$, if and only if the truncated data $[\delta_0,\delta_1, \ldots, \delta_{n-m+1}]$ is a solution of
(BCBP) for every $m=1,2,...,r$. They denote the maximum of all $r$ values, over all possible non-constant data, by $\Gamma(n)$ and called it {\it the gap}.  Certainly, non-trivial solutions of (BCBP) are not
that likely and so, it is  even less likely of instances when $\Gamma(n)>1$. It is actually conjectured (also in \cite{GR97}) that $\Gamma(n)\le 3$ for all $n\in \mathbb N$.
In a relatively recent paper (\cite{CohenShpilkaAvishay}), it is shown that $\Gamma(n)\le \sqrt{n}$ if $n=p^2-1$ with $p$ a prime number. 

In this paper we are basically only interested in the problem $\Gamma(n)>0$. Our Table~\ref{table:nontrivbisections} and partially the Table~\ref{table:nontrivbisections2} are extending the data from $n=128$ to $n=154$. However, it does not bring any new evidence in the way of non-trivial solutions but supporting  the conjecture that the trivial solutions are predominant.   We include a similar result to Theorem 2.6 (\cite{GR97}) in the case of odd $n$, with the property that $n$ and $(n+1)/2$ are primes. In the last section, we describe our algorithm that we use, in order  to solve the (BCBP)  for all $n\le 154$. Our algorithm is implemented in Maple and we run it on a usual laptop.

\section{Connection with the previous ``non-trivial" solution concept}

We would like to eliminate other solutions and replace the sequence $J_n$ with the sequence $\tilde J_n$, which is defined by:

\begin{quote}
    \noindent {\bf if $n$ even}, the number of all choices $\epsilon_{i}\in \{-1,0,1\}$, so that
\begin{equation}\label{neq1}
\sum_{i=0}^{n/2-1}\epsilon_{i}\binom{n}{i} =(-1)^{n/2+1} \frac{1}{2} \binom{n}{n/2},
\end{equation}
\noindent and $[\epsilon_1,\epsilon_2,\epsilon_3,...]$ is different of the trivial solution $s_0:=\underset{\frac{n}{2}}{[\underbrace{1,-1,1,-1,...}]}$.
\end{quote}

\begin{quote}
\noindent {\bf If $n$ is odd}, $\tilde J_n$ is the number of all choices $\epsilon_{i}\in \{-1,0,1\}$, so that
\begin{equation}\label{neq2}
\sum_{i=0}^{(n-1)/2}\epsilon_{i}\binom{n}{i} =0,
\end{equation}
\noindent where not all $\epsilon_{i}$, are equal to $0$; in this case we define $s_0:=\underset{\frac{n+1}{2}}{[\underbrace{0,0,0,...,0,0}]}$ the trivial solution
of (\ref{neq2}).
\end{quote}

It is clear that (\ref{neq2}) is invariant to the change $\epsilon_i\to -\epsilon_i$, which implies that
$\tilde J_{2n+1}$ is always even. To eliminate even this duplication we will work with

$$\hat J_n:=\begin{cases} \tilde J_n \ \ \text{if n is even} \\ \\
\tilde J_n/2 \ \ \text{if n is odd.}
\end{cases}
$$
By an abuse of notation we will sometime refer to $\hat J_n$ as the set of all the above solutions as in the following proposition.

\begin{proposition}\label{connectingthetwo} The connection between the two sequences is given by
$$J_n=\begin{cases} 2+2\underset{s\in \hat J_n}{\sum} 2^{m_s} \ \ \text{if n is even} \\ \\
2^{(n+1)/2}+2\underset{s\in \hat J_n}{\sum} 2^{m_s}  \ \ \text{if n is odd.}
\end{cases}$$
where $m_s$ is the number of $0$'s  in the vector $s$.
\end{proposition}
\proof  {\bf Case $n$ even:}  For every vector  $v=[v_0,v_2,...,v_{k-1}]$ with $k=n/2$ in $\hat J_n$ we have

$$\sum_{i=0}^{k-1}v_{i}\binom{n}{i} =(-1)^{k+1} \frac{1}{2} \binom{n}{k},\ \  \text{or}$$

$$\sum_{i=0}^{k-1}2 v_{i}\binom{n}{i}  +(-1)^{k} \binom{n}{k}=0 \Leftrightarrow $$

$$\sum_{i=0}^{n} w_{i}\binom{n}{i}=0, $$

\noindent where $w=[w_0,w_1,...,w_n]\in \{-1,1\}^{n+1}$, $w_i+w_{n-i}=2 v_i$, $i=0,...,k-1$ and $w_k=(-1)^{n/2}$.
We observe that for each $v_i=1$ or $v_i=-1$, the values of $w_i$ and $w_{n-i}$ are uniquely determined.  But for $v_i=0$,
there are two possible choices for $w_i$ and $w_{n-i}$, i.e., $\{(-1,1),(1,-1)\}$.  Hence for each $v$ we can find $2^{m_v}$
such possible $w$ which are solutions for the (BCBP), where $m_v$ is the number of zeros in $v$. The total number of solutions generated by $v$
have to be doubled since for every $w$ discussed above we can consider $-w$ which is also a solution for the (BCBP).
Since the trivial solution $s_0$ has no zeroes we obtain that $J_n=2+2\underset{s\in \hat J_n}{\sum} 2^{m_s}$.

  {\bf Case $n$ odd:} A similar argument can be employed  in this case. \eproof

\vspace{0.1in}

\noindent It is not difficult to see that

$$\hat J_{1}=\hat J_{2}=\tilde J_{3}=\hat J_{4}=\hat J_{5}=\hat J_{6}=\hat J_{7}=0,$$

\noindent and the first nonzero term in the sequence $\hat J_n$ is, as expected, $\hat J_8=1$.
We will illustrate a simple technique which is at the heart of our algorithm, in the special case   $\hat J_{14}=2$.

We have to look at the equation

\begin{equation}\label{eq14}\sum_{i=0}^6 \binom{14}{i}x_i=\frac{1}{2} \binom{14}{7}, \ \ x_i\in \{-1,0,1\}. \end{equation}

\n Let us observe that  a brute force solution for this little problem means to search through $3^7=2187$ possibilities. Instead, let us first notice that (\ref{eq14}) is equivalent to

$$x_0+14x_1+91x_2+364x_3+1001x_4+2002x_5+3003x_6=1716,  \ \ x_i\in \{-1,0,1\}.  $$

Projecting this modulo $7$, it implies that $x_0+6\equiv 0$ (mod $7$). The only option is $x_0=1$.
Taking (\ref{eq14}) modulo 13,  we have  $x_0+x_1\equiv 0$ (mod 13). This leads to only one option
$(x_0,x_1)\in \{(1,-1)\}$.  Projecting (\ref{eq14}) modulo 11 we obtain
$x_0+3x_1+3x_2+x_3\equiv 0$ (mod 11). Since $1+3+3+1<11$ this is equivalent to simply $x_0+3x_1+3x_2+x_3= 0$. Hence $x_0+x_3\equiv 0$ (mod 3) which basically implies $x_3=-x_0$.
Then $x_1+x_2=0$. This means we have only one solution $(x_0,x_1, x_2,x_3)\in\{(1,-1,1,-1)\} $.
The original equation  is then equivalent to $x_4+2x_5+3x_6=2$. From here we see that $x_5\equiv x_4+1$ (mod 3).
So we have the solutions $(x_4,x_5,x_6)\in\{(-1,0,1),(0,1,0),(1,-1,1)\}$. One of these gives the trivial solution $s_0$, and so
$\hat J_{14}=2$ with the non-trivial vector solutions $[1,-1,1,-1,-1,0,1]$ and $[1,-1,1,-1,0,1,0]$. By Proposition~\ref{connectingthetwo}, this makes $J_n=2+2(2+4)=12$.

\vspace{0.1in}
\section{Some more data}

Table~\ref{table:nontrivbisections} is an update of the results in \cite{IonandCo} and \cite{GR97}. It was  conjectured in \cite{IonandCo} that
$\hat J_{2^n}=1$ if $n$ is odd and $\hat J_{2^n}=0$ if $n$ is even. The numerical evidence still supports this conjecture.
In fact, a similar conjecture, we alluded to in the Introduction,  can be written in a more precise way    as

$$\lim_{n\to \infty} \frac{\#  \{   k|\hat J_{k}=0, 1\le k\le n \}}{n}=\frac{5}{6}.$$
The sequence $\{ n_k \}$ for which $\hat J_{n_k}>0$ behaves in a more or less expected way, although certainly chaotic.
In  Table~\ref{table:nontrivbisections} we codified the values of $n$ with the corresponding results in Theorem ~\ref{thm:nontr}.
One can arrive at other conjectures from this amount of data. For instance, we believe that on the other hand $\underset{k\to \infty} {\limsup} J_{n_k}=\infty$ although
this may seem to be in contradiction with the conjecture that $\Gamma(n)\le 3$.

For $\hat J_n>0$, the particular nontrivial solutions may give rise to new solutions. This is how we arrived at  the results of Theorem~\ref{thm:nontr}.  So, we have nontrivial solutions for $n=6k+2$, $k\in \mathbb N$. It is somewhat surprising that most of the time this is the only solution, i.e, $\hat J_{6k+2}=1$ for most $k$.  But if this is corroborated with other  situations in Theorem~\ref{thm:nontr}, we may see  $\hat J_n$ getting bigger. So far   $\hat J_{62}=8$ is the biggest value we encountered.
 These solutions and many others have something in common:
they contain a big number of alternating signs.

So, it is useful to employ the following identity:

\begin{equation}\label{combinaotrialid}
\sum_{j=0}^{\ell}(-1)^j{n \choose j}=(-1)^{\ell}{n-1\choose \ell},
\end{equation}

\noindent which can be proved easily by induction on $\ell$ and via the Pascal's identity ${n \choose \ell +1}={n -1\choose \ell +1}+{n -1\choose \ell }$.
Using this identity we can write the nontrivial solutions for $n=14$ in the following way:

$$-{13 \choose 3}-{14 \choose 4}+{14 \choose 6}=\frac{1}{2}{14 \choose 7} \ \text{and} \ \ -{13 \choose 3}+{14 \choose 5}=\frac{1}{2}{14 \choose 7}$$
 
The next theorem summarizes the various known (infinite) families of values of $n$ with non-trivial solutions to the (BCBP).
We include only proofs for parts (1) and (2) to give an idea of how to adapt to the new sequence $\hat J_n$.
\begin{thm}
\label{thm:nontr}
We have $\hat J_n >0$ for
\begin{enumerate}
\item   $\boxed{\heartsuit}$  $n\equiv 2$ (mod 6), (Theorem 3.6, \cite{GR97})
\item   $\boxed{\sharp}$ $n=4k^2-3, k\geq 2$ ($n=13, 33, 61, 97, 141, 193, \cdots $) (Theorem 3.10, \cite{GR97})
\item   $\boxed{\spadesuit}$  $n=4k^2-2, k\geq 2$, ($n=14, 34, 62, 98, 142, 194, \cdots $)(Theorem 3.8, \cite{GR97})
\item   $\boxed{\clubsuit}$ $5n^2+12n+8 = m^2$, ($n=14, 103, 713, 4894, 33551, ...$) (Theorem 3.9, \cite{GR97} or Theorem 12, \cite{IonandCo})
\item   $\boxed{\flat}$ $8n^2+1=m^2$, ($n=35, 1189,\cdots $) (Theorem 4.4, \cite{GR97})
\end{enumerate}\end{thm}

\begin{proof} $\bf \boxed{\heartsuit}$ For part 1, using the same technique as in  Introduction, we find that $\hat J_8=1$, with the unique nontrivial solution
$[1,-1,-1,0]$. We see the same pattern in one of the vector solutions of $\hat J_{14}$, namely $[1,-1,1,-1,-1,0,1]$. As pointed out above this can be written  as
$$-{13 \choose 3}-{14 \choose 4}+{14 \choose 6}=\frac{1}{2}{14 \choose 7} $$

\n via the identity  (\ref{combinaotrialid}). For $n=20$, we can write the non-trivial solution as

$$-{19 \choose 5}-{20 \choose 6}+{20 \choose 8}-{20 \choose 9}=-\frac{1}{2}{20 \choose 10} .$$

\n This suggests the general non-trivial solution if $n=6k+2$:

\begin{equation}\label{firstidentity}
\sum_{j=0}^{2k-1}(-1)^j{n \choose j}-{n \choose 2k}+\sum_{j=2k+2}^{3k}(-1)^j{n \choose j}=(-1)^{3k+2}\frac{1}{2}\binom{n}{n/2}.
\end{equation}
To prove this identity, we use (\ref{combinaotrialid}). Observe that (\ref{firstidentity}) is equivalent to
$$-{n-1 \choose 2k-1}-{n \choose 2k}+(-1)^{3k}{n-1 \choose 3k}-(-1)^{2k+1}{n-1 \choose 2k+1}=(-1)^{3k+2}\binom{n-1}{3k}.$$

\n After cancelations, this reduces to
$${n \choose 2k}={n-1 \choose 2k+1}-{n-1 \choose 2k-1} \Leftrightarrow n=6k+2. $$

$\bf \boxed{\sharp}$ For part 2,  since $n=4k^2-3$ we need to prove the identity

\begin{equation}\label{secondtidentity}
\binom{n}{s-3}-\binom{n}{s-2}-\binom{n}{s-1}+\binom{n}{s}=0,
\end{equation}
\n where $s=2k^2-k\le (n-1)/2$ ($k\ge 2$). Solving (\ref{secondtidentity}) for $s$ one obtains the quadratic $4s^2-4s(n-3)+(n-3)^2=n-3$
which has  the given solution ($n$ must be odd).

We leave the proofs for the rest of the cases for the interested reader, since they can be adapted to our scenario from \cite{GR97}. \end{proof}

   \begin{center}
   \begin{table} [ht]
\caption{Number of nontrivial Binomial Coefficients Bisections}{}
\vskip.2cm
  \allowdisplaybreaks[4]
  \centering
  \begin{tabular}{||c c | c c | c c|c c|c c|c c ||}
 \hline
 $n$ & $\hat J_n$ & $n$ & $\hat J_n$ & $n$ & $\hat J_n$& $n$ &  $\hat J_n$ & $n$ &  $\hat J_n$  & $n$ &  $\hat J_n$   \\ [0.5ex]
 \hline\hline
 1 & 0 & 25 &  $0$ &  49 &  0 &   \boxed{73} &   2   &   $\sharp$   $\boxed{97}$  &  1 &  121  &   0 \\
 \hline
 $\dagger$ 2 & 0 &  $\heartsuit$ \boxed{26} & $1$ & $\heartsuit$  \boxed{50}  &  1&  $\heartsuit$  \boxed{74} &  4 &   $\heartsuit$ $\spadesuit$ $\boxed{98}$ &   3 &     $\heartsuit$ $\boxed{122}$ &  1  \\
 \hline
 3 & 0 &     27  & $0$  & 51  &   0 &  75 &  0 &  99  &   0 & 123   &  0  \\
 \hline
 4 & 0 &  $\dagger$ 28 & $0$  & $\dagger$ 52   &   0 &   76 &  0 &  $\dagger$ 100 &   0  &   124  &  0\\
 \hline
 5 & 0 &  $\boxed{29}$  & $1$  & 53 &  0 &  77&    0  &     101&   0 &   125  &  0\\
 \hline
$\dagger$ 6 &  0 &  $\dagger$ 30  &   $0$  &  \boxed{54} & 1 &$\dagger$  78 & 0 &  $\dagger$ 102&    0&  $\dagger$  126  &  0 \\
 \hline
 7 & 0 &   \boxed{31}   &  $2$ &  55 & 0 &  79 &   0 &   $\clubsuit$  $\boxed{103}$ &    1&   127  &   0\\
 \hline
 $\heartsuit$ \boxed{8} & $\circled{1}$ &  $\heartsuit$   \boxed{ 32}  & $1$  &   $\heartsuit$ \boxed{56}   & 1  &   $\heartsuit$ \boxed{80} &  1 &   $\heartsuit$ \boxed{104}   &    2&   $\heartsuit$ \boxed{128}&    1\\
 \hline
 9 & 0  & $\sharp$ \boxed{33}  & $1$ & 57 & 0 &  81 &   0 &   105&   0&    129 &   0\\
 \hline
$\dagger$ 10 & 0 & $\spadesuit$ \boxed{34} & $\circled{5}$  &  $\dagger$ 58  & 0 &  $\dagger$ 82 &  0 &    $\dagger$106&  0 &   $\dagger$   130 &   0 \\
 \hline
 11 & 0  &  $\flat$ \boxed{35} & $2$ &  59 &  0&  83&  0&   107&     0&  131  &  0\\
 \hline
 $\dagger$ 12 & 0 &  $\dagger$ 36 & $0$  &   $\dagger$ 60&  0&   84&  0 &    $\dagger$ 108& 0 &  132 &    0\\
 \hline
$\sharp$  \boxed{13} &  $1$ &  37 & 0 & $\sharp$  61 &  1&  85 &   0&   109&    0&  133 &  0 \\
 \hline
$\heartsuit$ $\spadesuit$ $\clubsuit$ \boxed{ 14} & $\circled{2}$ &  $\heartsuit$ \boxed{38}  &  $2$ &  $\heartsuit$ $\spadesuit$ 62&  $\circled{8}$ &   $\heartsuit$ \boxed{86}  &   1&   $\heartsuit$ \boxed{110} &   1&     $\heartsuit$ \boxed{134} &   1\\
 \hline
 15 &  $0$ & 39 & $0$ & \boxed{63} & 1 &  87 &   0&  111 &    0&    135 &   0\\
 \hline
 $\dagger$16 & 0 & $\dagger$ 40 &  $0$ & 64 &  0 &  $\dagger$  88 &   0&   $\dagger$ 112&   0 &   $\dagger$  136 &   0\\
 \hline
 17 &  $0$ & \boxed{41} &  4  & 65 &  0&   89&  0&  113 &    0&     137&   0\\
  \hline
 $\dagger$ 18 & 0  &  $\dagger$ 42 &  0 & $\dagger$ 66 &  0&  90 &  0&   114&   0&  $\dagger$  138  &   0\\
  \hline
 19 &  $0$ & 43 &  0 & 67 &  0&  91 &   0&   115&   0 &    139 &   0\\
  \hline
 $\heartsuit$ \boxed{20} & $1$ & $\heartsuit$ \boxed{44} & 2  & $\heartsuit$  $ \boxed{68}$  &  1 & $\heartsuit$ \boxed{ 92} &  1&  $\heartsuit$ \boxed{ 116} &   1 &   $\heartsuit$ \boxed{ 140}&   1\\
 \hline
   21 & $0$ & 45 &   0 &   69 &  0 &  93 &  0&  117 &   0 &  $\sharp$   \boxed{141} &   1\\
 \hline
$\dagger$ 22 & $0$ & $\dagger$ 46 &  0 &  $\dagger$ 70&   0 &  94 &  0 &  118 &   0 &  $\spadesuit$\boxed{142}  &  1 \\
 \hline
  23 & $0$ & \boxed{47} &  1 & 71 &  0  & 95  &  0 &  119  &   0 &    143  & 0  \\
\hline
   \boxed{24} & $2$ & \boxed{48} &  1 & $\dagger$ 72 &  0 &   $\dagger$ 96 &   0&  120  &   0 &    144  & 0 \\  [1ex]
 \hline \hline
\end{tabular}
\label{table:nontrivbisections}
\end{table}
\end{center}

\noindent On the other hand, trivial solutions seemed to be predominant in the Table~\ref{table:nontrivbisections}.
We include here one information in this direction from \cite{GR97} (Theorem~2.6). For completion we include the proof of this here since it is also
in the spirit of our algorithm, but in a very special/fortunate  case. We search for a similar (infinite) sequence of odd numbers but all efforts turned unsuccessful.

\begin{thm}
\label{thm:trivial}
We have $\hat J_n =0$ for ($\dagger$) $n+1$ is an odd  prime.
\end{thm}

\begin{proof} We let $p=n+1$ and observe that  $n\equiv -1$ (mod $p$). First,  we show that $\binom{n}{j}\equiv (-1)^j$ (mod $p$), for every $j\in \{0,1,\ldots,n\}$.
This is clearly true for $j=0$.
Since, every $j\in \{1,\ldots,n\}$ has an inverse modulo $p$, we have for $j\in \{1,\ldots,n\}$
 
$$\binom{n}{j}  \equiv\frac{n(n-1)\cdots (n-j+1)}{j!}\\
 \equiv \frac{(-1)(-2)\cdots (-1-j+1)}{j!}\equiv (-1)^j\ \pmod p.$$

\n Then, if  $ [\delta_0,\ldots,\delta_n]$ is a solution of the (BCBP) problem
$$0=\sum_{j=0}^n \delta_j \binom{n}{j}\equiv \sum_{j=0}^n (-1)^j \delta_j:=\Delta \ \pmod p.$$
The number $\Delta=\sum_{j=0}^n (-1)^j \delta_j$ is an odd number ($p$ is an odd prime) satisfying

\begin{equation}\label{eq4}
|\Delta|\le \sum_{j=0}^n  |(-1)^j \delta_j|=\sum_{j=0}^n 1=n+1=p.
\end{equation}

 \n Since $\Delta$ cannot be zero, the only possible values of $\Delta$ are $p$ or $-p$.
Then the equality $|\Delta|=p= n+1$ in \eqref{eq4},  forces $\delta_j=\pm (-1)^j$, for all $j$. Therefore, we have only the two trivial solutions, that is,  $\hat J_n =0$.
\end{proof}

In the same spirit, we have the following result. 

\begin{thm}
\label{thm:solutionsforspecialn}
Suppose that $[\delta_0,\delta_1,\ldots,\delta_{n-1} ,\delta_n]$ is a solution of the (BCBP) for $n$ prime and $\frac{n+1}{2}$ also a prime
($n$ is in the sequence A005383). Then the folded sequence, $\eta=[\eta_0,\eta_,...,\eta_{(n-1)/2}]$ defined by $\eta_j=(-1)^{j}\frac{\delta_j+\delta_{n-j}}{2}$,
$j=0,1,2,...,(n-1)/2$ has the same number of $1$'s as the number of $-1$'s. \end{thm}

\begin{proof} We let $q=\frac{n+1}{2}$ and $p=2q-1=n$ be the two primes. The vector  $\eta$ satisfies 

\begin{equation}\label{eqnew}\sum_{j=0}^{\frac{n-1}{2}} (-1)^j \binom{n}{j}\eta_j=0.
\end{equation}

\n It is well known that $\binom{n}{j}$ are divisible by $n=p$ ($n$ is a prime number). This implies $\eta_0=0$. 
Dividing (\ref{eqnew}) by $p$, we obtain

\begin{equation}\label{eqnew2}\sum_{j=1}^{q-1} (-1)^j \frac{\binom{p}{j}}{p}\eta_j=0.
\end{equation}
Let us show that $\frac{\binom{p}{j}}{p}\equiv (-1)^{j-1}$ (mod $q$) for all $j=1,2,..., q-1$. Since $p\equiv -1$ (mod $q$) and for all $j$, $j=1,2,..., q-1$,
we have an inverse $j^{-1}$ modulo $q$ for $j$, we can write 
 
$$\frac{\binom{p}{j}}{p} \equiv\frac{(p-1)\cdots (p-j+1)}{j!}\\
\equiv \frac{(-2)(-3)\cdots (-1-j+1)}{j!}\equiv (-1)^{j-1}\ \pmod q.$$
 This shows that from (\ref{eqnew2}) we can conclude that $\sum_{j=1}^{q-1} \eta_j\equiv 0$ (mod $q$). This can clearly happen if and only if $\sum_{j=1}^{q-1} \eta_j=0$ which implies the statement in the
 theorem. 
\end{proof}
One may  perhaps use this result to provide some sort of positive answer to Q1 in the Introduction, by looking into a subsequence of A005383. Whether or not A005383 is infinite is 
just another open question. 

\begin{center}
\begin{table} [ht]
\caption{Number of nontrivial Binomial Coefficients Bisections (continued)}{}
\vskip.2cm
  \allowdisplaybreaks[4]
  \centering
  \begin{tabular}{||c c | c c | c c|c c|c c|c c ||}
 \hline
 $n$ & $\hat J_n$ & $n$ & $\hat J_n$ & $n$ & $\hat J_n$& $n$ &  $\hat J_n$ & $n$ &  $\hat J_n$  & $n$ &  $\hat J_n$   \\ [0.5ex]
 \hline\hline
 144 & 0 &  \maltese 155 &   &  $\dagger$ 166 &  $0$  &   177 &      &   $\heartsuit$ \boxed{188}  & $\ge 1$ &  199  &   0 \\
 \hline
 145 & 0 &  $\dagger$ 156 & $0$ & 167  &   &  $\dagger$ 178 &  $0$ &   189 &    &    $\heartsuit$ \boxed{200} & $\ge 1$   \\
 \hline
 $\heartsuit$ \boxed{146}& 1 &     157  &  & 168  &    &  179 &   & $\dagger$ 190  &   $0$ & 201   &    \\
 \hline
 147 & 0 &  $\heartsuit$ \boxed{158}& $\ge 1$  &  169  &    &   $\dagger$ 180 &  $0$ &   191 &     &   202  &  \\
 \hline
$\dagger$  148 & 0 &  159  &  & $\heartsuit$ \boxed{170} &   $\ge 1$ &  181 &     &   $\dagger$  192&   $0$ &   203  &  \\
 \hline
149 &  0 & 160  &   &  171 &  &  $\heartsuit$ \boxed{182}   & $\ge 1$ & $\sharp$   \boxed{193} &  $\ge 1$  & 204   &   \\
 \hline
$\dagger$ 150 & 0 &  161   &   & $\dagger$ 172 & 0 &  183 &    &  $\heartsuit$  $\spadesuit$\boxed{194} & $2$     &   205  &   \\
 \hline
151 & 0 & $\dagger$ 162 & $0$  &   173 &   & 184   &   &  195   &    &   $\heartsuit$ \boxed{206}  &  $\ge 1$   \\
 \hline
$\heartsuit$  152 & $ 1$  & 163 &  & 174 &  &  185 &    &   $\dagger$  196&   $0$&    207 &   0\\
 \hline
 153 & $0$ & $\heartsuit$ \boxed{164} & $\ge 1$ &  175  &  &  186 &   &     197&    &     208&   0 \\
 \hline
  154 &  $0$ &  165 &  &   $\heartsuit$ \boxed{176} &   $\ge 1$ &  187&  &    $\dagger$ 198&    $0$ &  209  &  0\\
  \hline \hline
 \end{tabular}
\label{table:nontrivbisections2}
\end{table}
\end{center}

\section{Non-trivial vector solutions}
Using our algorithm we extended the previous tables of non-trivial vector solutions.  Since our new non-trivial vectors differ,
in the way we record them, from the ones in \cite{IonandCo} and the Table 2 in (\cite{GR97}), we updated the previous lists.

\begin{center}
   \begin{table}
\caption{Nontrivial corresponding folded vectors}{}
\vskip.2cm
\allowdisplaybreaks[4]
\centering
\begin{tabular}{||  c | c ||}
 \hline
 $n$  & nontrivial  vectors \\ [0.5ex]
 \hline\hline
 8  & [1,-1,-1, 0]\\
 \hline
 13  & [0,0,0,1,-1,-1,1]\\ [1ex]
 \hline
 14  & [1,-1,1,-1,-1,0,1]\\
     & [1,-1,1,-1,0,1,0]\\
 \hline
 20  & [1,-1,1,-1,1,-1,-1,0,1,-1]\\
 \hline
 24  & [1,-1,-1,-1,1,0,-1,0,1,0,-1,0],\\
    & [-1,1,-1,0,1,1,-1,0,-1,0,1,-1] \\
 \hline
 26  & [1,-1,1,-1,1,-1,1,-1,-1,0,1,-1,1] \\
 \hline
 29  & [0,0,0,0,0,0,0,1,-1,0,1,1,-1,0,0]\\
 \hline
 31  & [0,0,0,1,-1,0,0,-1,-1,0,-1,1,1,-1,-1,1]\\
      & [0,0,0,1,-1,0,0,-1,-1,1,0,1,0,1,0,-1]\\
 [1ex]\hline \hline
\end{tabular}
\label{table: solutions1}
\end{table}
\end{center}

{\small
\begin{center}
 \begin{tabular}{||  c | c ||}
 \hline
 $n$  & nontrivial  folded vectors \\ [0.5ex]
 \hline
\hline 
 32  & [1,-1,1,-1,1,-1,1,-1,1,-1,-1,0,1,-1,1,-1]\\
 33  & [0,0,0,0,0,0,0,0,0,0,0,0,1,-1,-1,1,0]\\
 34  &[1,-1,1,-1,1,-1,-1,1,1,0,-1,0,0,-1,0,0,1]\\
      &[1,-1,1,-1,1,-1,-1,1,1,0,-1,0,0,1,0,-1,1]\\
      &[1,-1,1,-1,1,-1,-1,1,1,0,1,1,-1,-1,0,0,1]\\
       &[1,-1,1,-1,1,-1,-1,1,1,0,1,1,-1,1,0,-1,1]\\
       &[1,-1,1,-1,1,-1,1,-1,1,-1,1,-1,1,0,-1,0,1]\\
35   &[1,-1,1,-1,1,-1,1,-1,1,-1,1,1,0,-1,-1,1,0,0]\\
      &[1,-1,1,-1,1,-1,1,-1,1,-1,1,-1,1,-1,-1,1,0,0]\\
       \hline
38   &[1,-1,1,-1,1,-1,1,-1,1,-1,1,-1,-1,0,1,-1,1,-1,1]\\
     &[1,-1,1,1,1,0,1,1,-1,-1,1,-1,0,0,1,1,1,-1,0]\\
 \hline
 41   &[0,0,0,0,0,1,-1,1,1,1,0,-1,0,-1,1,-1,-1,1,0,0,0]\\
       &[0,0,0,0,0,1,-1,1,1,1,0,-1,0,1,0,-1,-1,1,0,0,0]\\\
       &[0,0,0,0,0,1,0,1,0,0,0,1,-1,-1,1,-1,0,-1,1,0,0]\\
      &[0,0,0,0,0,1,0,1,0,0,0,1,-1,1,0,-1,0,-1,1,0,0]\\
 \hline
 44    &[1,-1,1,-1,1,-1,1,-1,1,-1,1,-1,1,-1,-1,0,1,-1,1,-1,1,-1]\\
    &[1,-1,1,-1,1,-1,1,-1,1,-1,1,-1,1,-1,-1,0,1,-1,1,-1,1,-1]\\
     \hline
47 & [0,0,0,0,0,0,0,0,0,0,0,1,0,0,-1,0,-1,-1,1,0,0,0,0,0]\\
 \hline
 48  & [-1,1,-1,0,1,1,0,0,0,-1,1,0,0,1,0,0,1,0,1,0,0,-1,1,-1]\\
 \hline
50  & [1,-1,1,-1,1,-1,1,-1,1,-1,1,-1,1,-1,1,-1,-1,0,1,-1,1,-1,1,-1,1]\\
\hline
 54   & [1,-1,1,-1,1,-1,1,-1,1,-1,1,-1,1,-1,0,1,1,0,1,-1,0,1,0,-1,1,-1,1]\\
 \hline
 61  & [0,0,0,0,0,0,0,0,0,0,0,0,0,0,0,0,0,0,0,0,0,0,0,0,0,1,-1,-1,1,0,0]\\
 62  & [1,-1, 1,-1,1,-1,1,-1,1,-1,1,-1,1,-1,1,-1,1,-1,1,-1,1,-1,1,-1,1,-1,0,1,0,-1,1]\\
     & [1,-1,1,-1,1,-1,1,-1,1,-1,1,-1,1,-1,1,-1,1,-1,1,-1,-1,0,1,-1,1,-1,1,-1,1,-1,1] \\
     & [1,-1,1,-1,1,-1,1,-1,1,-1,1,-1,1,-1,1,-1,1,-1,1,-1,-1,0,1,-1,1,-1,0,1,0,-1,1] \\
     & [1,-1,1,-1,-1,1,1,-1,1,-1,1,0,-1,0,-1,0,1,1,-1,-1,-1,-1,0,1,1,1,-1,-1,0,0,1] \\
     & [1,-1,1,-1,-1,1,1,-1,1,-1,0,0,1,1,0,-1,-1,-1,0,0,1,0,-1,-1,0,1,0,-1,0,0,1]\\
     & [1,-1,1,-1,-1,1,1,-1,1,-1,0,0,1,1,0,-1,-1,-1,0,0,1,0,-1,-1,0,1,-1,1,-1,0,1]\\
     & [1,-1,1,-1,-1,1,1,-1,1,-1,0,0,1,1,0,-1,-1,-1,0,0,-1,1,-1,-1,0,1,0,-1,0,0,1]\\
     &[1,-1,1,-1,-1,1,1,-1,1,-1,0,0,1,1,0,-1,-1,-1,0,0,-1,1,-1,-1,0,1,-1,1,-1,0,1] \\
      \hline
 63  &[1,-1,1,-1,1,1,-1,0,1,0,-1,0,-1,1,-1,1,-1,0,-1,1,1,1,-1,1,1,-1,-1,1,1,-1,0,0]\\
 \hline
 68  &[1,-1,1,-1,1,-1,1,-1,1,-1,1,-1,1,-1,1,-1,1,-1,1,-1,1,-1,-1,0,1,-1,1,-1,1,-1,1,-1,1,-1]\\
 \hline
 73  &[0,0,0,0,0,0,0,0,0,0,0,0,0,0,0,1,0,-1,0,0,-1,-1,-1,-1,1,0,-1,1,-1,-1,1,0,0,0,0,0,0] \\
  &[0,0,0,0,0,0,0,0,0,0,0,0,0,0,0,1,0,1,1,1,0,0,1,0,-1,0,1,0,-1,-1,1,0,0,0,0,0,0]\\
\hline
 74  & [1,-1,1,-1,1,-1,1,-1,1,-1,1,-1,1,-1,1,-1,-1,0,1,1,1,1,-1,1,0,0,0,0,-1,0,1,-1,1,-1,1,-1,1]\\
 &[1,-1,1,-1,1,-1,1,-1,1,-1,1,-1,1,-1,1,-1,0,-1,1,1,1,0,-1,0,0,1,-1,0,0,1,0,-1,1,-1,1,-1,1]\\
 &[1,-1,1,-1,1,-1,1,-1,1,-1,1,-1,1,-1,1,-1,0,0,-1,0,-1,1,-1,0,0,1,-1,0,0,1,0,-1,1,-1,1,-1,1]\\
 &[1,-1,1,-1,1,-1,1,-1,1,-1,1,-1,1,-1,1,-1,1,-1,1,-1,1,-1,1,-1,-1,0,1,-1,1,-1,1,-1,1,-1,1,-1,1]\\
\hline
 80 &[1,-1,1,-1,1,-1,1,-1,1,-1,1,-1,1,-1,1,-1,1,-1,1,-1,1,-1,1,-1,1,-1,-1,0,1,-1,1,-1,1,-1,1,-1,1,-1,1,-1]\\
 \hline
 86 & [1,-1,1,-1,1,-1,1,-1,1,-1,1,-1,1,-1,1,-1,1,-1,1,-1,1,-1,1,-1,1,-1,1,-1,-1,0,1,-1,1,-1,1,-1,1,-1,1,-1,1,-1,1]\\
 \hline
 92  & [1,-1,1,-1,1,-1,1,-1,1,-1,1,-1,1,-1,1,-1,1,-1,1,-1,1,-1,1,-1,1,-1,1,-1,1,-1,-1,0,1,-1,1,-1,1,-1,1,-1,1,-1,1,-1,1,-1]\\
 \hline
 97  &[0,0,0,0,0,0,0,0,0,0,0,0,0,0,0,0,0,0,0,0,0,0,0,0,0,0,0,0,0,0,0,0,0,0,0,0,0,0,0,0,0,0,1,-1,-1,1,0,0,0]\\
 \hline
 98  &{\small [1,-1,1,-1,1,-1,1,-1,1,-1,1,-1,1,-1,1,-1,1,-1,1,-1,1,-1,1,-1,1,-1,1,-1,1,-1,1,-1,-1,0,1,-1,1,-1,1,-1,1,-1,1,-1,1,-1,1,-1,1]}\\
      &{\small  [1,-1,1,-1,1,-1,1,-1,1,-1,1,-1,1,-1,1,-1,1,-1,1,-1,1,-1,1,-1,1,-1,1,-1,1,-1,1,-1,-1,0,1,-1,1,-1,1,-1,1,-1,1,0,-1,0,1,-1,1]}\\
      &{\small [1,-1,1,-1,1,-1,1,-1,1,-1,1,-1,1,-1,1,-1,1,-1,1,-1,1,-1,1,-1,1,-1,1,-1,1,-1,1,-1,1,-1,1,-1,1,-1,1,-1,1,-1,1,0,-1,0,1,-1,1]}\\
\hline \hline
\end{tabular}
\label{table: solutions2}
\end{center}}
The rest of the non-trivial solution vectors for $n\in\{103,104,110, 116,122, 128,134, 140,141\}$ are only the expected ones.

\section{The algorithm implemented in a Maple Program}

The main idea is to reduce the problem to a number of congruencies of smaller number of variables. The second idea is to make the sequence of variables
nested, i.e., the set of variables at step $k$ is included in the set of variables at step $k+1$.
First let us introduce the following sequence
$$D_i:=\gcd \{  \binom {n}{j} : j=i,i+1,...,\lfloor \frac{n}{2}\rfloor \}, \ \ \ i=0,1,2,...,\lfloor \frac{n}{2}\rfloor .
$$
Clearly $\{D_i\}$ is a non-decreasing sequence. This sequence can be constant for some values of $i$, say $D_s=D_{s+1}=...=D_{t}$,
for some $s,t$ ($s<t$). For that reason we remove duplications and redefine $\{d_i\}$ strictly increasing such that
$\{d_1,d_2,...\}=\{D_2,D_3, ...\}$ ($D_1=1$ so we eliminate this trivial situation). We add to the list $d:=[d_1,d_2,...]$ the value  $2^n$, to make sure the last congruency is in fact equivalent to one of the equalities (\ref{neq1}) or (\ref{neq2}).   We solve then the equation (\ref{neq1}) or (\ref{neq2}) modulo $d_i$,

$$\sum_{j=0}^{m} \binom{n}{j} x_j\equiv c_i\ (mod\ \  d_i) ,m=\begin{cases} n/2-1\ \text{ if n is even} \\ (n-1)/2 \ \ \text{if ne is odd}  \\  \end{cases}$$

\n for the variables involved and the solutions obtained are carried into the next equation (mod $d_{i+1}$).
At each step we need to solve for a relatively small number of variables. The number of solutions at each step, say $\{s_i\}$, has an interesting distribution.
$$\underset{Figure\ 2,\  \text{The sequence} \ \{s_i\} \ \text{for} \ n=62   } {\epsfig{file=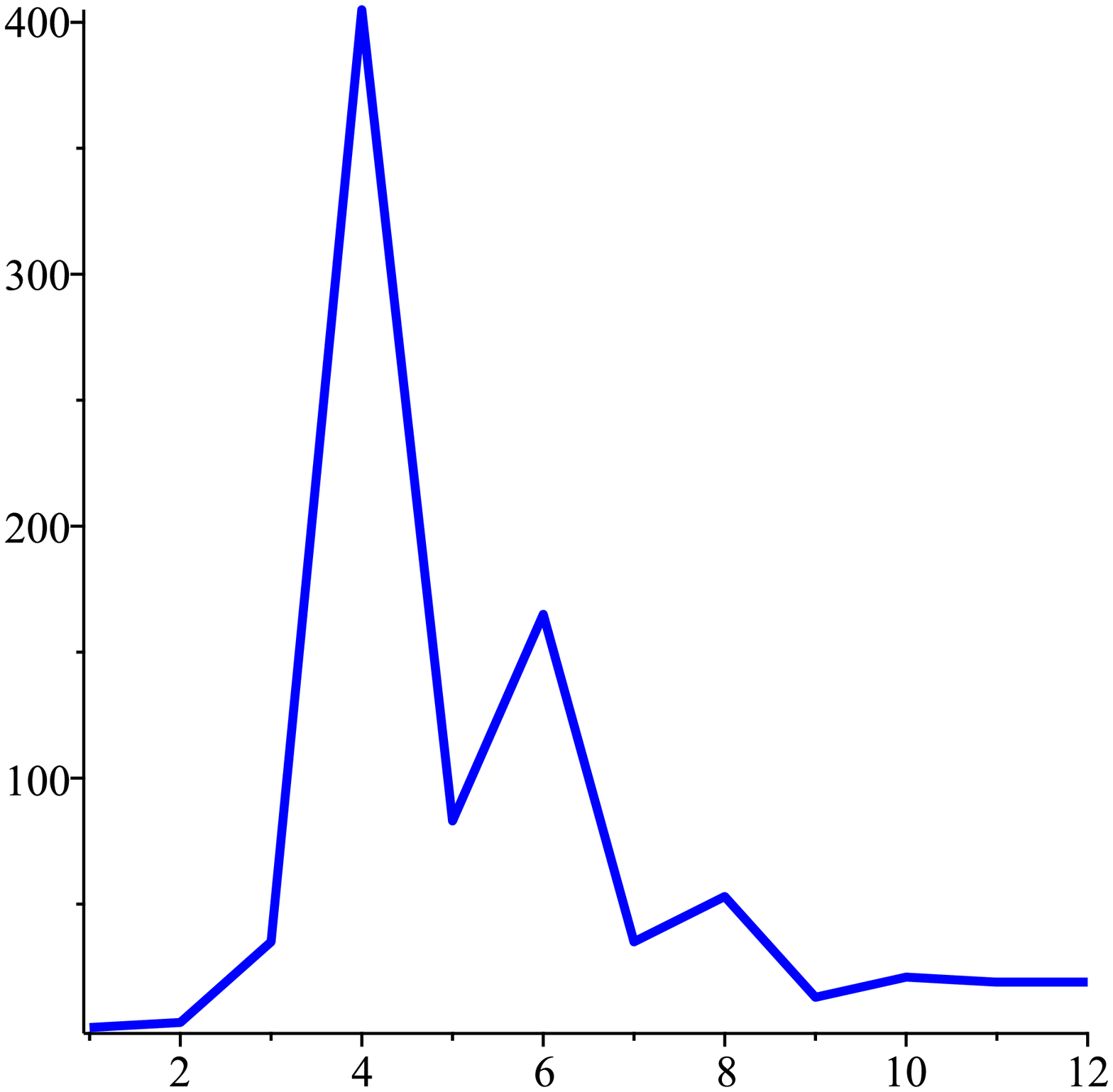,height=2in,width=3in} }$$
Since every solution of  (\ref{neq1}) or (\ref{neq2}) should also satisfy the above congruencies, it is clear that in the end we get all the solutions of  (\ref{neq1}) or (\ref{neq2}).

We are going to exemplify our algorithm in the case $n=19$ which is again a very fortunate case.
The equation (\ref{neq2}) becomes

\begin{equation}\label{eq19}
\begin{array}{c}
x_0+19x_1+171x_2+969x_3+3876x_4+11628x_5+27132x_6+\\ \\
50388x_7+75582x_8+92378x_9=0, \ \ \ \ x_i\in \{-1,0,1\}.
\end{array}
\end{equation}

\n The list  $d=[d_1,d_2,...]$ is $[19, 323, 646, 8398, 92378, 524288=2^{19}]$. So, the first step we get (\ref{eq19}) modulo 19:
$$ x_0=0 \ \ (mod\  19),$$
\n which has clearly only the trivial solution $x_0=0$.  Substituting this solution in  (\ref{eq19}) and taking everything modulo $323=(19)(17)$, we obtain

$$19x_1+171x_2\equiv 0 \ \ (mod \ 323) \Leftrightarrow  x_1+9x_2\equiv 0 \ \ (mod\  17).$$
Obviously, this last congruency has only the trivial solution again, i.e., $x_1=x_2=0$.
We substitute into (\ref{eq19}) and taking everything modulo $646=(19)(17)(2)$, gives
$$323x_3\equiv 0 \ \  (mod\ 646) \Leftrightarrow x_3\equiv 0\ \  (mod\ \  2).$$
This forces $x_3=0$ and then we move to the next step and use the modulo  $8398=(2)(13)(17)(19)$:
$$3876x_4+3230x_5+1938x_6 \equiv 0 \ \ (mod \ 8398) \Leftrightarrow 6x_4+5x_5+3x_6 \equiv 0 \ \ (mod \ 13).$$
Again, we do not have any nontrivial solutions and so we move on to modulo $92378=(2)(11)(13)(17)(19)$:
$$50388x_7+75582x_8\equiv 0 \ \ (mod\ 92378)\ \Leftrightarrow 6x_7+9x_8\equiv 0 \ \ (mod\ 11).$$
Since we do get anything non-trivial we conclude that $x_9$ must be zero also, and so we have only the trivial solution in this case.
We observe that $s_i=1$ for all $i$. This is what makes this situation so special.

The program can be found at http://ejionascu.ro/notes/program.pdf. To estimate the complexity of this algorithm one needs to have a good control on the sequence 
$\{D_i\}$ which is described with certain precision in \cite{jorisOS}. It is clear that if one can uniformly bound the number of new variables at each step, then the complexity becomes $O(n)$. It is surprising that the time required to run the program is not linear in terms of $n$. We found that
it has a big oscillating behavior.   For instance, we could run it for $n=194$ in a few minutes but it takes hours for $n=155$.

\end{document}